\documentclass[11pt]{article}
\usepackage[]{amsmath,amssymb}
\usepackage[]{amsthm,enumerate}
\usepackage{verbatim,bm}

\usepackage{color}
\usepackage[]{comment}
\usepackage{a4wide}
\usepackage{lscape,mathdots}

\newtheorem{theorem}{Theorem}[section]
\newtheorem{proposition}[theorem]{Proposition}
\newtheorem{lemma}[theorem]{Lemma}

\theoremstyle{definition}
\newtheorem{definition}[theorem]{Definition}

\newcommand{\defn}[1]{{\em #1}}
\theoremstyle{remark}
\newtheorem{remark}[theorem]{Remark}

\newcommand{\DW}{\text{DW}}
\newcommand{\OD}{\text{OD}}

\makeatletter
\@addtoreset{equation}{section}

\makeatother

\title{Disjoint weighing matrices}
\date{
\today
}
\author{
 Hadi Kharaghani\thanks{Department of Mathematics and Computer Science, University of Lethbridge,
Lethbridge, Alberta, T1K 3M4, Canada. \texttt{kharaghani@uleth.ca}}
\and
 Sho Suda\thanks{Department of Mathematics, National Defense Academy of Japan,
2-10-20 Hashirimizu, Yokosuka, Kanagawa, 239-8686, Japan. \texttt{ssuda@nda.ac.jp}}
\and
Behruz Tayfeh-Rezaie\thanks{School of Mathematics, Institute for Research in Fundamental Sciences (IPM), P.O. Box 19395-5746, Tehran, Iran. \texttt{tayfeh-r@ipm.ir}}
}

\begin{document}
\maketitle
{\centering\footnotesize Dedicated to professor Arasu on the occasion of his 65th birthday\par}
%%%%%%%%%%%%%%%%%%%%%%%%%%%%%%%%%%%%%%%%%%%%%%%%%%%%%%%%
%%%%%%%%%%%%%%%%%%%%%%%%%%%%%%%%%%%%%%%%%%%%%%%%%%
\abstract{
The notion of disjoint weighing matrices is introduced as a generalization of orthogonal designs. A recursive construction along with a computer search lead to 
some infinite classes of disjoint weighing matrices, which in turn are shown to form commutative association schemes with 3 or 4 classes.
}

\section{Introduction}
A {\em weighing matrix of order $n$ and weight $w$}, denoted $W(n,w)$, is an $n\times n$ $(0,1,-1)$-matrix $W$ such that $W W^\top=w I_n$, where $I_n$ denotes the identity matrix of order $n$.  
An {\em orthogonal design of order $n$ and type $(w_1,\ldots,w_k)$ in variables $x_1,\ldots,x_k$}, denoted by $OD(n;w_1,\ldots,w_k)$,  is a square $(0,\pm x_1,\ldots,\pm x_k)$-matrix $D$, where $x_1,\ldots,x_k$ are distinct commuting indeterminates, such that $$D D^\top=(w_1x_1^2+\cdots+w_k x_k^2)I_n.$$
Splitting the matrix $D=\sum_{i=1}^k x_i W_i$, where each $W_i$ is a weighing matrix $W(n,w_i)$, then $\sum_{i=1}^k W_i$ is a $(0,1,-1)$-matrix.

Furthermore, the weighing matrices $W_i$ satisfy the {\it antiamicability} condition  $W_iW_j^\top+W_jW_i^\top=O_n$ for any distinct pair $i,j$, where $O_n$ denotes the zero matrix of order $n$. The antiamicability condition imposes restrictions on the number of disjoint weighing matrices stemming from orthogonal designs. Relaxing this condition provides a much larger class of disjoint weighing matrices enjoying very nice properties. This has motivated us to study these matrices.
For a $(0,1,-1)$-matrix $X$, denote by $|X|$ the matrix obtained by replacing $-1$ with $1$ in $X$.

\begin{definition}
Let $n,k,w_1,\ldots,w_k$ be positive integers and let $W_i$ be a weighing matrix of order $n$ and weight $w_i$ for $i\in\{1,\ldots,k\}$.
If  $W_i$ are disjoint, that is $\sum_{i=1}^k |W_i|$ is a $(0,1)$-matrix, then we call $W_1,\ldots,W_k$ \defn{disjoint weighing matrices} denoting it by $DW(n;w_1,\ldots,w_k)$.
\end{definition}
The variety of applications of disjoint weighing matrices includes  their use in the construction of Bush-type Hadamard matrices \cite{K}, their role as the building blocks of orthogonal designs and their use  in the construction of symmetric group divisible designs  and association schemes \cite{KS19}. Craigen \cite{c} used these matrices in his study of disjoint weighing designs.

The notation $\DW(n;[\frac{n-1}{k}]^k)$ will be used to show $k$ disjoint weighing matrices of  order $n$ and weight $\frac{n-1}{k}$. The main focus of the paper is on the  skew-symmetric disjoint weighing matrices $\DW(km+1;[{m}]^k)$ whose sum is the complete design $J_{km+1}-I_{mk+1}$, where $J_n$ is the all one matrix of order $n$.
.
In Section~\ref{sec:2}, we construct skew-symmetric $\DW(n;[\frac{n-1}{3}]^3)$ for $n\in\{7\cdot 4^m,10\cdot 4^m,17\cdot 4^m\}$ and also $\DW(2^{mn},[\frac{2^{mn}-1}{2^{n}-1}]^{2^n-1})$ by using the Goethals-Seidel arrays and a new recursive method.

Imprimitive association schemes with few classes have been constructed
from design theoretic objects such as symmetric designs, linked systems
of symmetric designs \cite{D}, \cite{M}, those of symmetric group divisible designs \cite{KS2017}, \cite{KS19}, and
spherical designs \cite{DGS}.
In Section 3, we make use of the constructed disjoint skew weighing matrices and specific Hadamard matrices to 
offer a generalization of  symmetric group divisible designs. 
Furthermore, we construct  commutative association schemes with $3$ or $4$ classes from this.

%%%%%%%%%%%%%%%%%%%%%%%%%%%%%%%%%%%%%%%%%%%%%%%%%%%%%%%%%%%%%%%%%%%%%%%%%%%%%%%%%%%%%%%%%%%%
\section{A recursive construction}\label{sec:2}
For circulant $n\times n$ $(0,1,-1)$-matrices $A,B,C,D$, a \emph{Goethals-Seidel array} is
\begin{align*}
\begin{pmatrix}
A   & BR & CR & DR \\
-BR &  A & D^\top R & -C^\top R \\
-CR & -D^\top R & A & B^\top R \\
-DR & C^\top R & -B^\top R & A
\end{pmatrix},
\end{align*}
where $R$ is the back identity matrix of order $n$.
If $A$ is skew-symmetric and $AA^\top+BB^\top+CC^\top+DD^\top=w I_n$, then the Goethals-Seidel array provides a skew-symmetric weighing matrix of order $4n$ and weight $w$. Note that  for $w=n-1$, it must be that  $n$ is  odd since for even $n$,   skew-symmetric matrix $A$ has at least two zeros in each row.  We use this array to construct skew-symmetric $\DW(n;[\frac{n-1}{3}]^3)$ for some small values of $n$ through a computer search. This can be done only
when $n$ is 4 congruent modulo $12$.  
\begin{lemma}\label{lem:dw28}
There exist skew-symmetric $\DW(28;9,9,9)$ and $\DW(52;17,17,17)$.
\begin{proof}
The following row vectors $a_i,b_i,c_i,d_i$ ($i=1,2,3$) are the first rows of circulant matrices being used in Goethals-Seidel array in order to construct three skew-symmetric weighing matrices of order $28$ and weight $9$:
\begin{align*}
a_1&=(0,1,0,0,0,0,-1), &a_2&=(0,0,0,1,-1,0,0),&a_3&=(0,0,1,0,0,-1,0), \\
b_1&=(0,0,0,1,0,1,0), &b_2&=(0,0,0,0,0,0,1), &b_3&=(-1,1,1,0,1,0,0),\\
c_1&=(0,0,0,0,0,1,0), &c_2&=(-1,1,1,0,1,0,0), &c_3&=(0,0,0,1,0,0,1), \\
d_1&=(-1,1,1,0,1,0,0), &d_2&=(0,0,0,0,0,1,1),&d_3&=(0,0,0,1,0,0,0).
\end{align*}

For the other case, we use the following row vectors $a_i,b_i,c_i,d_i$ ($i=1,2,3$) as the first rows of circulant matrices being used in Goethals-Seidel array in order to construct three skew-symmetric weighing matrices of order $52$ and weight $17$:
\begin{align*}
a_1&=(0,  0,  0,  0,  0,  0,  0,  0,  0,  0,  0,  0,  0), &b_1&=(1, -1, -1, -1, -1,  0, -1,  0,  0,  0,  0,  0,  0),\\
c_1&=(1, -1,  1,  1, -1, -1,  0,  0,  0,  0,  0,  0,  0),&d_1&=(0,  0,  0,  0, -1,  0,  0,  1, -1,  0, -1,  0,  1),\\
a_2&=(0,  1,  0,  1,  1, -1,  0,  0,  1, -1, -1,  0, -1),&b_2&=(0,  0,  0,  0,  0, -1,  0,  0,  1, -1,  0, -1,  0),\\
c_2&=(0,  0,  0,  0,  0,  0, -1,  0,  0, -1, -1,  0,  0),&d_2&=(0,  0,  0,  0,  0,  0, -1,  0,  0,  0,  0, -1,  0),\\
a_3&=(0,  0,  1,  0,  0,  0,  1, -1,  0,  0,  0, -1,  0),&b_3&=(0,  0,  0,  0,  0,  0,  0, -1,  0,  0, -1,  0, -1),\\
c_3&=(0,  0,  0,  0,  0,  0,  0, -1,  1,  0,  0, -1, -1),&d_3&=(1,  1,  1, -1,  0,  1,  0,  0,  0, -1,  0,  0,  0).
\end{align*}
\end{proof}
\end{lemma}

\begin{remark}
 There is no compelling reason that an $\OD(52;17,17,17)$ may exist and an exhaustive computational search to find it requires a long time mainly due to the antiamicability requirements of the three disjoint weighing matrices. In order to search for any orthogonal design the first step is to find the corresponding disjoint weighing matrices. The computer search for three disjoint weighing matrices was feasible and the three disjoint $W(52,17)$ were found.  We believe that there is a $\DW(3k+1;k,k,k)$ for every positive integer $k\equiv 1\pmod{4}$. \end{remark}

By \cite{HKT}, an $\OD(40;13,13,13)$ exists which results in  the following.

\begin{lemma}\label{lem:dw40}
There exist skew-symmetric $\DW(40;13,13,13)$.
\end{lemma}

\begin{comment}
Define $H_i,K_i$ by
\begin{align}
H_1&=H_2=\begin{pmatrix}
1 & 1 & 1 & 1 \\
1 & -1 & 1 & -1 \\
1 & 1 & -1 & -1 \\
1 & -1 & -1 & 1
\end{pmatrix},
H_3=\begin{pmatrix}
1 & 1 & 1 & -1 \\
1 & -1 & -1 & -1 \\
1 & -1 & 1 & 1 \\
-1 & -1 & 1 & -1
\end{pmatrix}, \nonumber \\
K_1&=\begin{pmatrix}
0 & 0 & -1 & 0 \\
0 & 0 & 0 & -1 \\
1 & 0 & 0 & 0 \\
0 & 1 & 0 & 0
\end{pmatrix},
K_2=\begin{pmatrix}
0 & -1 & 0 & 0 \\
1 & 0 & 0 & 0 \\
0 & 0 & 0 & -1 \\
0 & 0 & 1 & 0
\end{pmatrix},
K_3=\begin{pmatrix}
0 & 0 & 0 & 1 \\
0 & 0 & -1 & 0 \\
0 & 1 & 0 & 0 \\
-1 & 0 & 0 & 0
\end{pmatrix}. \label{eq;K}
\end{align}
Then its is easy to see that $H_i$ is a symmetric Hadamard matrix, $K_i$ is skew-symmetric, and $H_i K_i^\top=K_i H_i^\top$ for any $i\in\{1,2,3\}$.
\end{comment}

%In the sequel, we use  $J_n$ to denote the all one matrix of order $n$.
We prepare the following lemma for a recursive construction.
\begin{lemma}\label{lem:dw2}
For any positive integer $m$, there exist symmetric Hadamard matrices $H_1,\ldots,H_{2^m-1}$ of order $2^m$ and skew-symmetric signed permutation matrices $K_1,\ldots,K_{2^m-1}$ of order $2^m$ satisfying
\begin{enumerate}
\item $H_i K_i^\top=K_i H_i^\top$ for $i\in\{1,\ldots,2^m-1\}$, and
\item $\sum_{i=1}^{2^m-1}|K_i|=J_{2^m}-I_{2^m}$.
\end{enumerate}
\end{lemma}
\begin{proof}
For $m=1$, the desired matrices are
\begin{align*}
H=\begin{pmatrix} 1 & 1 \\ 1 & -1 \end{pmatrix},\quad K=\begin{pmatrix} 0 & 1 \\ -1 & 0 \end{pmatrix}.
\end{align*}

For $m\geq2$, we prove the statement by induction on $m$ simultaneously with the following fact  that there exist symmetric Hadamard matrices $L_1,\ldots,L_{2^m-1}$ of order $2^m$ and symmetric signed permutation matrices $M_1,\ldots,M_{2^m-1}$ of order $2^m$ satisfying
\begin{enumerate}[(a)]
\item $L_i M_i^\top=M_i L_i^\top$ for $i\in\{1,\ldots,2^m-1\}$, and
\item $\sum_{i=1}^{2^m-1}|M_i|=J_{2^m}-I_{2^m}$.
\end{enumerate}

For $m=2$, let $Q=\begin{pmatrix} 1 & 0 \\ 0 & -1 \end{pmatrix}$ and $R=J_2-I_2$. Then
$(H_i,K_i)$ ($i=1,2,3$) are
\begin{align*}
(H\otimes H,K\otimes I),(H\otimes H,I\otimes K),(I_2\otimes (Q+R)+R\otimes (Q-R),R\otimes K),
\end{align*}
and $(L_i,M_i)$ ($i=1,2,3$) are
\begin{align*}
(I_2\otimes (Q+R)+R\otimes (Q-R),R\otimes I),((Q+R)\otimes I_2+(Q-R)\otimes R ,I\otimes R),(H\otimes H,K\otimes K).
\end{align*}

Assume that the case of $m$ is true, namely there exist pairs $(H_i,K_i)$ and $(L_i,M_i)$ ($i\in\{1,\ldots,2^m-1\}$) satisfying (1), (2) and (a), (b) respectively.
Then the following are pairs of symmetric Hadamard matrices and skew-symmetric signed permutation matrices satisfying (1) and (2):
\begin{itemize}
\item $(H \otimes H_j ,I_2\otimes K_j)$ for $j\in\{1,\ldots,2^m-1\}$,
\item $(H\otimes L_j, K\otimes M_j)$ for $j\in\{1,\ldots,2^m-1\}$,
\item $(H^{\otimes (m+1)},K\otimes I_{2^m})$,
\end{itemize}
and the following are pairs of symmetric Hadamard matrices and symmetric signed permutation matrices with diagonals $0$ satisfying (a) and (b):
\begin{itemize}
\item $(H \otimes H_j ,K\otimes K_j)$ for $j\in\{1,\ldots,2^m-1\}$,
\item $(H\otimes L_j, I\otimes M_j)$ for $j\in\{1,\ldots,2^m-1\}$,
\item $((I\otimes(Q+R)+R\otimes(Q-R))\otimes H^{\otimes m},R\otimes I_{2^m})$.
\end{itemize}
This completes the proof.
\end{proof}

We are now ready to present a recursive construction for  skew-symmetric disjoint weighing matrices.
\begin{proposition}\label{prop:rc}
Let  $W_1,\ldots,W_k$ be skew-symmetric $\DW(km+1;[m]^k)$. Let $H_i,K_i$ ($i\in\{1,\ldots,k\}$) be such that each $H_i$ is a symmetric Hadamard matrix of order $k+1$ and  each $K_i$ is a skew-symmetric signed permutation matrix of order $k+1$ satisfying $H_i K_i^\top=K_i H_i^\top$ and $\sum_{i=1}^k |K_i|=J_{k+1}-I_{k+1}$. Then there exist skew-symmetric $\DW((k+1)(km+1);[(k+1) m+1]^k)$.
\end{proposition}
\begin{proof}

Define $$\tilde{W}_i=H_i\otimes W_i+K_i\otimes I_{km+1}$$ for $i\in\{1,\ldots,k\}$.
By the definition of $H_i,W_i,K_i$, the matrix $\tilde{W}_i$ is a $(0,1,-1)$-matrix.
For $i\in\{1,\ldots,k\}$,
\begin{align*}
\tilde{W}_i^\top=H_i^\top\otimes W_i^\top+K_i^\top\otimes I_{km+1}^\top=H_i\otimes (-W_i)+(-K_i)\otimes I_{km+1}=-\tilde{W}_i.
\end{align*}
Thus $\tilde{W}_i$ is skew-symmetric.
Next we calculate $\tilde{W}_i\tilde{W}_i^\top$:
\begin{align*}
\tilde{W}_i\tilde{W}_i^\top&=H_iH_i^\top\otimes W_iW_i^\top+H_iK_i^\top\otimes W_i+K_iH_i^\top\otimes W_i^\top+K_iK_i^\top\otimes I_{km+1}\\
&=(k+1) m I_{k+1}\otimes I_{km+1}+H_iK_i^\top\otimes (W_i+W_i^\top)+I_{k+1}\otimes I_{km+1}\\
&=((k+1)m+1)I_{(k+1)(km+1)},
\end{align*}
which implies that $\tilde{W}_i$ is a weighing matrix of order $(k+1)(km+1)$ and weight $(k+1) m+1$.
Finally, we have 
\begin{align*}
\sum_{i=1}^k |\tilde{W}_i|&=\sum_{i=1}^k (|H_i|\otimes |W_i|+|K_i|\otimes I_{km+1})\\
&= J_{k+1}\otimes \sum_{i=1}^k|W_i|+\sum_{i=1}^k|K_i|\otimes I_{km+1}\\
&= J_{k+1}\otimes (J_{km+1}-I_{km+1})+(J_{k+1}-I_{k+1})\otimes I_{km+1}\\
&= J_{(k+1)(km+1)}-I_{(k+1)(km+1)}
\end{align*}
as desired.
\end{proof}

Applying Proposition~\ref{prop:rc} to  the constructed  skew-symmetric disjoint weighing matrices in Lemmas~\ref{lem:dw28}, \ref{lem:dw40} and \ref{lem:dw2} recursively, we obtain the following.
\begin{theorem} \label{thm2nm}
For positive integers $m$ and $n$ such that $n\geq 2$, the following exist:
\begin{enumerate}
\item skew-symmetric $\DW(2^{mn};[\frac{2^{mn}-1}{2^n-1}]^{2^n-1})$.
\item skew-symmetric $\DW(7\cdot 4^{m};[\frac{7\cdot 4^{m}-1}{3}]^3)$.
\item skew-symmetric $\DW(10\cdot 4^{m};[\frac{10\cdot 4^{m}-1}{3}]^3)$.
\item skew-symmetric $\DW(13\cdot 4^{m};[\frac{13\cdot 4^{m}-1}{3}]^3)$.
\end{enumerate}
\end{theorem}
\medskip

\begin{remark}
\begin{enumerate}
 \item By taking $m\ge 2$ in part 1 of the above theorem it is observed that the number of disjoint weighing matrices by far exceeds the number of disjoint weighing matrices obtained from an orthogonal design  of the same order.
 \item The orthogonal designs $\OD(7\cdot 4^{m};[\frac{7\cdot 4^{m}-1}{3}]^3)$,  $\OD(10\cdot 4^{m};[\frac{10\cdot 4^{m}-1}{3}]^3)$, and $\OD(13\cdot 4^{m};[\frac{13\cdot 4^{m}-1}{3}]^3)$ are not known to exist for any $m>1$.
 \item Note that part (1) is obtained by considering the family $K_1,K_2,\cdots,K_{2^n-1}$ from Lemma 2.4. 
\end{enumerate}
\end{remark}

%%%%%%%%%%%%%%%%%%%%%%%%%%%%%%%%%%%%%%%%%%%%%%%%%%%%%%%%%%%%%%%%%%%%%%%%%%%%%%%%%%%%%%%%%%%%%
\section{Association schemes}\label{sec:3}
We make use of skew symmetric disjoint weighing matrices with Hadamard matrices to obtain association schemes.

A  \emph{(commutative) association scheme of $d$ classes}
with vertex set $X$ of size $n$
is a set of non-zero $(0,1)$-matrices $A_0, \ldots, A_d$, which are called {\em adjacency matrices}, with
rows and columns indexed by $X$, such that:
\begin{enumerate}
\item $A_0=I_n$.
\item $\sum_{i=0}^d A_i = J_n$.
\item $\{A_1^\top,\ldots,A_d^\top\}=\{A_1,\ldots,A_d\}$.
\item For any $i,j\in\{0,1,\ldots,d\}$, $A_iA_j=\sum_{k=0}^d p_{ij}^k A_k$
for some $p_{ij}^k$'s.
\item For any $i,j\in\{1,\ldots,d\}$, $A_iA_j=A_j A_i$.
\end{enumerate}

The vector space over $\mathbb{R}$ spanned by $A_i$'s forms a commutative algebra, denoted by $\mathcal{A}$ and called the  \emph{Bose-Mesner algebra}.
There exists a basis of $\mathcal{A}$ consisting of primitive idempotents, say $E_0=(1/n)J_n,E_1,\ldots,E_d$.
Since  $\{A_0,A_1,\ldots,A_d\}$ and $\{E_0,E_1,\ldots,E_d\}$ are two bases of $\mathcal{A}$, there exist the change-of-bases matrices $P=(P_{ij})_{i,j=0}^d$, $Q=(Q_{ij})_{i,j=0}^d$ so that
\begin{align*}
A_j=\sum_{i=0}^d P_{ij}E_i,\quad E_j=\frac{1}{n}\sum_{i=0}^d Q_{ij}A_i.
\end{align*}
The matrix $P$ ($Q$ respectively) is said to be the {\em first (second respectively) eigenmatrix}.
See \cite{BI} for details.

Let $k,\ell,m$ be positive integers such that there exist a Hadamard matrix $H$ of order $k \ell +1$ and skew-symmetric $\DW(km+1;[m]^k)$ $W_1,\ldots,W_k$.
Let $W_{i,1},W_{i,2}$ be disjoint $(0,1)$-matrices such that $W_i=W_{i,1}-W_{i,2}$.
Since $W_i$ is a weighing matrix of weight $m$,
\begin{align}\label{eq:w0}
W_{i,1}W_{i,2}+W_{i,2}W_{i,1}-W_{i,1}^2-W_{i,2}^2=m I.
\end{align}

Let $H$ be normalized, that is, the first row of $H$ is the all one row vector.
For $i\in\{1,\ldots,k\ell \}$, let $C_i$ be the auxiliary $(1,-1)$-matrix corresponding to the $(i+1)$-th  row of $H$, that is $C_i=r_{i+1}^\top r_{i+1}$ where $r_{i+1}$ is the $(i+1)$-th row of $H$.
Decompose $C_i$ into disjoint $(0,1)$-matrices $D_{i,1},D_{i,2}$ defined as $C_i=D_{i,1}-D_{i,2}$.
The following are fundamental properties. We omit their routine proof.
\begin{lemma}\label{lem:d1}
\begin{enumerate}
\item For $i\in\{1,\ldots,k \ell\}$ and $j\in\{1,2\}$, $D_{i,j}^\top=D_{i,j}$ holds.
\item For $j\in\{1,2\}$, $\sum_{i=1}^{k \ell}D_{i,j}=\frac{k \ell+1}{2} I_{k \ell+1}+\frac{k \ell-1}{2}J_{k \ell+1}$ holds.
\item For $i\in\{1,\ldots,k \ell\}$ and $j\in\{1,2\}$, $D_{i,j}^2=\frac{k \ell+1}{2} D_{i,1}$ holds.
\item For $i\in\{1,\ldots,k \ell\}$ and distinct $j,j'\in\{1,2\}$, $D_{i,j}D_{i,j'}=\frac{k \ell+1}{2} D_{i,2}$ holds.
\item For distinct $i,i'\in\{1,\ldots, k \ell\}$ and $j,j'\in\{1,2\}$, $D_{i,j} D_{i',j'}=\frac{k\ell+1}{4}J_{k\ell+1}$ holds.
\item For $i\in\{1,\ldots,k\ell\}$ and $j\in\{1,2\}$, $D_{i,j} J_{k\ell+1}=J_{k\ell+1} D_{i,j}=\frac{k\ell+1}{2}J_{k\ell+1}$ holds.
\end{enumerate}
\end{lemma}

For square matrices $X_1,\ldots,X_m$ of the same size, define the back-circulant matrix with the first row blocks $X_1,\ldots,X_m$ as
$$
\text{b-circ}(X_1,X_2,\ldots,X_m)=\begin{pmatrix}
X_1 & X_2 & \cdots & X_{m-1} & X_m \\
X_2 & & X_{m-1} & X_m & X_1 \\
\vdots & \iddots & X_m & X_1 & \vdots \\
X_{m-1} & \iddots & \iddots &  & X_{m-2} \\
X_m & X_1 & \cdots & X_{m-2} & X_{m-1}
\end{pmatrix}.
$$
For $i\in\{1,\ldots,k\}$ and $j\in\{1,2\}$, set $B_{i,j}=\text{b-circ}(D_{1+(i-1)\ell,j},D_{2+(i-1)\ell,j},\ldots,D_{\ell+(i-1)\ell,j})$.
Note that each $B_{i,j}$ is symmetric.
Then the following is an easy consequence of Lemma~\ref{lem:d1}.
\begin{lemma}\label{lem:b1}
\begin{enumerate}
\item For $i\in\{1,\ldots,k\}$ and $j\in\{1,2\}$, $B_{i,j}^2=\frac{k \ell+1}{2} I_{\ell}\otimes \sum_{h=1}^{\ell}D_{h+(i-1)\ell,1}+\frac{\ell(k\ell+1)}{4}(J_{\ell}-I_{\ell})\otimes J_{k\ell+1}$ holds.
\item For $i\in\{1,\ldots,k\}$ and distinct $j,j'\in\{1,2\}$, $B_{i,j}B_{i,j'}=\frac{k \ell+1}{2} I_{\ell}\otimes \sum_{h=1}^{\ell}D_{h+(i-1)\ell,2}+\frac{\ell(k\ell+1)}{4}(J_{\ell}-I_{\ell})\otimes J_{k\ell+1}$ holds.
\item For distinct $i,i'\in\{1,\ldots,k\}$ and $j,j'\in\{1,2\}$, $B_{i,j}B_{i',j'}=\frac{\ell(k\ell+1)}{4}J_{\ell}\otimes J_{k\ell+1}$ holds.
\end{enumerate}
\end{lemma}

Define $(0,1)$-matrices $A_1,A_2$ by
\begin{align*}
A_1=\sum_{i=1}^k (W_{i,1}\otimes B_{i,1}+W_{i,2}\otimes B_{i,2}),\quad
A_2=\sum_{i=1}^k (W_{i,1}\otimes B_{i,2}+W_{i,2}\otimes B_{i,1}).
\end{align*}
Note that in the case $\ell=1$, $A_1$ and $A_2$ appeared in \cite{KS19} as an example of symmetric group divisible designs.
Then $A_1+A_2=(J_{km+1}-I_{km+1})\otimes J_{\ell}\otimes J_{k\ell+1}$ and
\begin{align*}
A_1^\top=\sum_{i=1}^k (W_{i,1}^\top\otimes B_{i,1}^\top+W_{i,2}^\top\otimes B_{i,2}^\top)
=\sum_{i=1}^k (W_{i,2}\otimes B_{i,1}+W_{i,1}\otimes B_{i,2})
=A_2.
\end{align*}
Define symmetric $(0,1)$-matrices $A_0,A_3,A_4$ by
\begin{align*}
A_0&= I_{km+1}\otimes I_{\ell}\otimes I_{k\ell+1},\\
A_3&= I_{km+1}\otimes I_{\ell}\otimes (J_{k\ell+1}-I_{k\ell+1}),\\
A_4&= I_{km+1}\otimes (J_{\ell}-I_{\ell})\otimes J_{k\ell+1}.
\end{align*}
Note that $A_4=O$ if and only if $\ell=1$.
\begin{theorem}\label{thm:as1}
\begin{enumerate}
\item If $\ell>1$, then the matrices $A_0,\ldots,A_4$ form a commutative association scheme with $4$ classes.
\item If $\ell=1$, then the matrices $A_0,\ldots,A_3$ form a commutative association scheme with $3$ classes.
\end{enumerate}
\end{theorem}
\begin{proof}
We prove (1) and (2) simultaneously. It is enough to show that $\mathcal{A}=\text{span}\{A_0,A_1,\ldots,A_4\}$ is closed commutative under the matrix multiplication. Namely, we show that
\begin{align}\label{eq:as1}
A_iA_j=A_jA_i\in\mathcal{A} \text{ for } i,j\in\{1,\ldots,4\}.
\end{align}
For $i,j\in\{3,4\}$, \eqref{eq:as1} is easy to see.
For $i\in\{1,2\},j\in\{3,4\}$ or $i\in\{3,4\},j\in\{1,2\}$, it suffices to see that 
\begin{align*}
A_i\cdot I_{km+1}\otimes I_{\ell}\otimes J_{k\ell+1}&=I_{km+1}\otimes I_{\ell}\otimes J_{k\ell+1}\cdot  A_i\in\mathcal{A}, \\
A_i\cdot I_{km+1}\otimes J_{\ell}\otimes J_{k\ell+1}&=I_{km+1}\otimes J_{\ell}\otimes J_{k\ell+1} \cdot A_i\in\mathcal{A}
\end{align*}

for $i\in\{1,2\}$.
By Lemma~\ref{lem:d1} (6), we have
\begin{align*}
B_{i,j}\cdot I_{\ell}\otimes J_{k\ell+1}&= I_{\ell}\otimes J_{k\ell+1} \cdot B_{i,j}= \frac{k\ell+1}{2}J_{\ell}\otimes J_{k\ell+1}, \\
B_{i,j}\cdot J_{\ell}\otimes J_{k\ell+1}&= J_{\ell}\otimes J_{k\ell+1} \cdot B_{i,j}=\frac{\ell(k\ell+1)}{2}J_{\ell}\otimes J_{k\ell+1}.
\end{align*}
Then
\begin{align*}
A_1\cdot I_{km+1}\otimes I_{\ell}\otimes J_{k\ell+1}&=\sum_{i=1}^k (W_{i,1}\otimes (B_{i,1}\cdot I_{\ell}\otimes J_{k\ell+1})+W_{i,2}\otimes (B_{i,2}\cdot I_{\ell}\otimes J_{k\ell+1}))\\
&=\sum_{i=1}^k (W_{i,1}\otimes \frac{k\ell+1}{2}J_{\ell}\otimes J_{k\ell+1}+W_{i,2}\otimes \frac{k\ell+1}{2}J_{\ell}\otimes J_{k\ell+1})\\
&=\frac{k\ell+1}{2}\sum_{i=1}^k (W_{i,1}+W_{i,2})\otimes J_{\ell}\otimes J_{k\ell+1}\\
&=\frac{k\ell+1}{2}(J_{km+1}-I_{km+1})\otimes J_{\ell}\otimes J_{k\ell+1}\\
&\in \mathcal{A}.
\end{align*}
Also, $I_{km+1}\otimes I_{\ell}\otimes J_{k\ell+1}\cdot A_1$ satisfies the same equation and the other cases also follow from the same calculation.

Finally we show the case $i,j\in\{1,2\}$.
By Lemma~\ref{lem:b1} (1), (2), (3),
\begin{align}\label{eq:w1}
(A_1)^2&=\sum_{i,i'=1}^{k}\sum_{j,j'=1}^2W_{i,j}W_{i',j'}\otimes B_{i,j} B_{i',j'}\nonumber \displaybreak[0]\\
&=\sum_{i=1}^{k}\sum_{j=1}^2 W_{i,j}^2\otimes B_{i,j}^2+\sum_{i=1}^{k}\sum_{j\neq j'}W_{i,j}W_{i,j'}\otimes B_{i,j} B_{i,j'}+\sum_{i\neq i'}\sum_{j,j'=1}^2W_{i,j}W_{i',j'}\otimes B_{i,j} B_{i',j'}\nonumber \displaybreak[0]\\
&=\sum_{i=1}^k\sum_{j=1}^2 W_{i,j}^2\otimes \left(\frac{k \ell+1}{2} I_{\ell}\otimes \sum_{h=1}^{\ell}D_{h+(i-1)\ell,1}+\frac{\ell(k\ell+1)}{4}(J_{\ell}-I_{\ell})\otimes J_{k\ell+1}\right)\nonumber \\
&\quad +\sum_{i=1}^k\sum_{j\neq j'} W_{i,j}W_{i,j'}\otimes \left(\frac{k \ell+1}{2} I_{\ell}\otimes \sum_{h=1}^{\ell}D_{h+(i-1)\ell,2}+\frac{\ell(k\ell+1)}{4}(J_{\ell}-I_{\ell})\otimes J_{k\ell+1}\right)\nonumber \\
&\quad +\sum_{i\neq i'} \sum_{j,j'=1}^2 W_{i,j}W_{i',j'}\otimes\left( \frac{\ell(k\ell+1)}{4}J_{\ell}\otimes J_{k\ell+1}\right)\displaybreak[0]\nonumber \\
&=\frac{k \ell+1}{2} \sum_{i=1}^k\sum_{j=1}^2 W_{i,j}^2\otimes I_{\ell}\otimes \sum_{h=1}^{\ell}D_{h+(i-1)\ell,1}+\frac{k \ell+1}{2} \sum_{i=1}^k\sum_{j\neq j'} W_{i,j}W_{i,j'}\otimes I_{\ell}\otimes \sum_{h=1}^{\ell}D_{h+(i-1)\ell,2}\nonumber \\
&\quad +\frac{\ell(k\ell+1)}{4}\sum_{i=1}^k\sum_{j,j'=1}^2 W_{i,j}W_{i,j'}\otimes (J_{\ell}-I_{\ell})\otimes J_{k\ell+1}\nonumber \\
&\quad +\frac{\ell(k\ell+1)}{4}\sum_{i\neq i'} \sum_{j,j'=1}^2 W_{i,j}W_{i',j'}\otimes J_{\ell}\otimes J_{k\ell+1}\nonumber \displaybreak[0]\\
&=\frac{k \ell+1}{2} \sum_{i=1}^k (W_{i,1}^2+W_{i,2}^2-W_{i,1}W_{i,2}-W_{i,2}W_{i,1})\otimes I_{\ell}\otimes \sum_{h=1}^{\ell}D_{h+(i-1)\ell,1}\nonumber \\
&\quad +\frac{\ell(k \ell+1)}{2} \sum_{i=1}^k\sum_{j\neq j'} W_{i,j}W_{i,j'}\otimes I_{\ell}\otimes J_{k\ell+1}\nonumber \\
&\quad +\frac{\ell(k\ell+1)}{4}\sum_{i=1}^k\sum_{j,j'=1}^2 W_{i,j}W_{i,j'}\otimes (J_{\ell}-I_{\ell})\otimes J_{k\ell+1}\nonumber \\
&\quad +\frac{\ell(k\ell+1)}{4}\sum_{i\neq i'} \sum_{j,j'=1}^2 W_{i,j}W_{i',j'}\otimes J_{\ell}\otimes J_{k\ell+1}\displaybreak[0].
\end{align}

For the first terms in \eqref{eq:w1}, by \eqref{eq:w0} and Lemma~\ref{lem:d1} (2),
\begin{align}\label{eq:w2}
&\sum_{i=1}^k (W_{i,1}^2+W_{i,2}^2-W_{i,1}W_{i,2}-W_{i,2}W_{i,1})\otimes I_{\ell}\otimes \sum_{h=1}^{\ell}D_{h+(i-1)\ell,1}\nonumber \\
&=-mI_{km+1}\otimes I_{\ell}\otimes \sum_{i=1}^k\sum_{h=1}^{\ell}D_{h+(i-1)\ell,1}\nonumber \\
&=-mI_{km+1}\otimes I_{\ell}\otimes\left(\frac{k\ell+1}{2}I_{k\ell+1}+\frac{k\ell-1}{2}J_{k\ell+1}\right)
\end{align}
and for the last three terms in \eqref{eq:w1}, by \eqref{eq:w0} and the fact that $W_{i,1}+W_{i,2}=J_{km+1}-I_{km+1}$,
\begin{align}\label{eq:w3}
&2\sum_{i=1}^k\sum_{j\neq j'} W_{i,j}W_{i,j'}\otimes I_{\ell}+\sum_{i=1}^k\sum_{j,j'=1}^2 W_{i,j}W_{i,j'}\otimes (J_{\ell}-I_{\ell})+\sum_{i\neq i'} \sum_{j,j'=1}^2 W_{i,j}W_{i',j'}\otimes J_{\ell}\nonumber \\
&=\sum_{i=1}^k(2\sum_{j\neq j'} W_{i,j}W_{i,j'}-\sum_{j,j'=1}^2 W_{i,j}W_{i,j'})\otimes I_{\ell}+(\sum_{i=1}^k\sum_{j,j'=1}^2 W_{i,j}W_{i,j'}+\sum_{i\neq i'} \sum_{j,j'=1}^2 W_{i,j}W_{i',j'})\otimes J_{\ell}\nonumber \\
&=\sum_{i=1}^k(W_{i,1}W_{i,2}+W_{i,2}W_{i,1}-W_{i,1}^2-W_{i,2}^2)\otimes I_{\ell}+\sum_{i,i'=1}^k\sum_{j,j'=1}^2 W_{i,j}W_{i',j'}\otimes J_{\ell}\nonumber \\
&=kmI_{km+1}\otimes I_{\ell}+\sum_{i=1}^k(W_{i,1}+W_{i,2})\sum_{i'=1}^k(W_{i',1}+W_{i',2})\otimes J_{\ell}\nonumber \\
&=kmI_{km+1}\otimes I_{\ell}+(J_{km+1}-I_{km+1})^2\otimes J_{\ell}\nonumber \\
&=kmI_{km+1}\otimes I_{\ell}+I_{km+1}\otimes J_{\ell}+(km-1)J_{km+1}\otimes J_{\ell}.
\end{align}
Putting \eqref{eq:w2} and \eqref{eq:w3} into \eqref{eq:w1}, \eqref{eq:w1} implies that 
\begin{align*}
(A_1)^2=&-\frac{m(k \ell+1)^2}{4}I_{km+1}\otimes I_{\ell}\otimes I_{k\ell+1}+\frac{m(k \ell+1)}{4}I_{km+1}\otimes I_{\ell}\otimes J_{k\ell+1}\\
&+\frac{\ell(k\ell+1)}{4}I_{km+1}\otimes J_{\ell}\otimes J_{k\ell+1}+\frac{\ell(k\ell+1)(km-1)}{4}J_{km+1}\otimes J_{\ell}\otimes J_{k\ell+1}.
\end{align*}
The same equation is true for $A_2$.

In the following, the second index $j,j'$ are considered in modulo $2$.
By Lemma~\ref{lem:b1} (1), (2), (3),
\begin{align}\label{eq:w4}
A_1A_2&=\sum_{i,i'=1}^{k}\sum_{j,j'=1}^2W_{i,j}W_{i',j'}\otimes B_{i,j} B_{i',j'+1}\nonumber \displaybreak[0]\\
&=\sum_{i=1}^{k}\sum_{j\neq j'} W_{i,j}W_{i,j'}\otimes B_{i,j}^2+\sum_{i=1}^{k}\sum_{j=1}^2 W_{i,j}^2\otimes B_{i,j} B_{i,j+1}+\sum_{i\neq i'}\sum_{j,j'=1}^2W_{i,j}W_{i',j'}\otimes B_{i,j} B_{i',j'+1}\nonumber \displaybreak[0]\\
&=\sum_{i=1}^k\sum_{j\neq j'} W_{i,j}W_{i,j'}\otimes \left(\frac{k \ell+1}{2} I_{\ell}\otimes \sum_{h=1}^{\ell}D_{h+(i-1)\ell,1}+\frac{\ell(k\ell+1)}{4}(J_{\ell}-I_{\ell})\otimes J_{k\ell+1}\right)\nonumber \\
&\quad +\sum_{i=1}^k\sum_{j=1}^2 W_{i,j}^2\otimes \left(\frac{k \ell+1}{2} I_{\ell}\otimes \sum_{h=1}^{\ell}D_{h+(i-1)\ell,2}+\frac{\ell(k\ell+1)}{4}(J_{\ell}-I_{\ell})\otimes J_{k\ell+1}\right)\nonumber \\
&\quad +\sum_{i\neq i'} \sum_{j,j'=1}^2 W_{i,j}W_{i',j'}\otimes\left( \frac{\ell(k\ell+1)}{4}J_{\ell}\otimes J_{k\ell+1}\right)\displaybreak[0]\nonumber \\
&=\frac{k \ell+1}{2} \sum_{i=1}^k\sum_{j\neq j'} W_{i,j}W_{i,j'}\otimes I_{\ell}\otimes \sum_{h=1}^{\ell}D_{h+(i-1)\ell,1}+\frac{k \ell+1}{2} \sum_{i=1}^k\sum_{j=1}^2 W_{i,j}^2\otimes I_{\ell}\otimes \sum_{h=1}^{\ell}D_{h+(i-1)\ell,2}\nonumber \\
&\quad +\frac{\ell(k\ell+1)}{4}\sum_{i=1}^k\sum_{j,j'=1}^2 W_{i,j}W_{i,j'}\otimes (J_{\ell}-I_{\ell})\otimes J_{k\ell+1}\nonumber \\
&\quad +\frac{\ell(k\ell+1)}{4}\sum_{i\neq i'} \sum_{j,j'=1}^2 W_{i,j}W_{i',j'}\otimes J_{\ell}\otimes J_{k\ell+1}\nonumber \displaybreak[0]\\
&=\frac{k \ell+1}{2} \sum_{i=1}^k (W_{i,1}W_{i,2}+W_{i,2}W_{i,1}-W_{i,1}^2-W_{i,2}^2)\otimes I_{\ell}\otimes \sum_{h=1}^{\ell}D_{h+(i-1)\ell,1}\nonumber \\
&\quad +\frac{\ell(k \ell+1)}{2} \sum_{i=1}^k\sum_{j=1}^2 W_{i,j}^2\otimes I_{\ell}\otimes J_{k\ell+1}\nonumber \\
&\quad +\frac{\ell(k\ell+1)}{4}\sum_{i=1}^k\sum_{j,j'=1}^2 W_{i,j}W_{i,j'}\otimes (J_{\ell}-I_{\ell})\otimes J_{k\ell+1}\nonumber \\
&\quad +\frac{\ell(k\ell+1)}{4}\sum_{i\neq i'} \sum_{j,j'=1}^2 W_{i,j}W_{i',j'}\otimes J_{\ell}\otimes J_{k\ell+1}\displaybreak[0].
\end{align}

For the first terms in \eqref{eq:w4}, by \eqref{eq:w0} and Lemma~\ref{lem:d1} (2),
\begin{align}\label{eq:w5}
&\sum_{i=1}^k (W_{i,1}W_{i,2}+W_{i,2}W_{i,1}-W_{i,1}^2-W_{i,2}^2)\otimes I_{\ell}\otimes \sum_{h=1}^{\ell}D_{h+(i-1)\ell,1}\nonumber \\
&=mI_{km+1}\otimes I_{\ell}\otimes \sum_{i=1}^k\sum_{h=1}^{\ell}D_{h+(i-1)\ell,1}\nonumber \\
&=mI_{km+1}\otimes I_{\ell}\otimes\left(\frac{k\ell+1}{2}I_{k\ell+1}+\frac{k\ell-1}{2}J_{k\ell+1}\right)
\end{align}
and for the last three terms in \eqref{eq:w4}, by \eqref{eq:w0} and the fact that $W_{i,1}+W_{i,2}=J_{km+1}-I_{km+1}$,
\begin{align}\label{eq:w6}
&2\sum_{i=1}^k\sum_{j=1}^2 W_{i,j}^2\otimes I_{\ell}+\sum_{i=1}^k\sum_{j,j'=1}^2 W_{i,j}W_{i,j'}\otimes (J_{\ell}-I_{\ell})+\sum_{i\neq i'} \sum_{j,j'=1}^2 W_{i,j}W_{i',j'}\otimes J_{\ell}\nonumber \\
&=\sum_{i=1}^k(2\sum_{j=1}^2 W_{i,j}^2-\sum_{j,j'=1}^2 W_{i,j}W_{i,j'})\otimes I_{\ell}+(\sum_{i=1}^k\sum_{j,j'=1}^2 W_{i,j}W_{i,j'}+\sum_{i\neq i'} \sum_{j,j'=1}^2 W_{i,j}W_{i',j'})\otimes J_{\ell}\nonumber \\
&=\sum_{i=1}^k(W_{i,1}^2+W_{i,2}^2-W_{i,1}W_{i,2}-W_{i,2}W_{i,1})\otimes I_{\ell}+\sum_{i,i'=1}^k\sum_{j,j'=1}^2 W_{i,j}W_{i',j'}\otimes J_{\ell}\nonumber \\
&=-kmI_{km+1}\otimes I_{\ell}+\sum_{i=1}^k(W_{i,1}+W_{i,2})\sum_{i'=1}^k(W_{i',1}+W_{i',2})\otimes J_{\ell}\nonumber \\
&=-kmI_{km+1}\otimes I_{\ell}+(J_{km+1}-I_{km+1})^2\otimes J_{\ell}\nonumber \\
&=-kmI_{km+1}\otimes I_{\ell}+I_{km+1}\otimes J_{\ell}+(km-1)J_{km+1}\otimes J_{\ell}.
\end{align}
Putting \eqref{eq:w5} and \eqref{eq:w6} into \eqref{eq:w4}, \eqref{eq:w4} gives
\begin{align*}
A_1A_2=&\frac{m(k \ell+1)^2}{4}I_{km+1}\otimes I_{\ell}\otimes I_{k\ell+1}-\frac{m(k \ell+1)}{4}I_{km+1}\otimes I_{\ell}\otimes J_{k\ell+1}\\
&+\frac{\ell(k\ell+1)}{4}I_{km+1}\otimes J_{\ell}\otimes J_{k\ell+1}+\frac{\ell(k\ell+1)(km-1)}{4}J_{km+1}\otimes J_{\ell}\otimes J_{k\ell+1}.
\end{align*}
 The same equation is true for $A_2A_1$. This concludes that $\mathcal{A}$ is closed and commutative under the matrix multiplication.
\end{proof}

We determine the eigenmatrices depending on whether $\ell>1$ or $\ell=1$.

(1) The intersection matrix $L_1=(p_{1,j}^k)_{j,k=0}^{4}$ is
\begin{align*}
B_1=\left(
\begin{array}{ccccc}
 0 & 1 & 0 & 0 & 0 \\
 0 & \frac{1}{4} \ell (k \ell+1) (k m-1) & \frac{1}{4} \ell (k \ell+1) (k m-1) & \frac{1}{4} m (k \ell+1)^2 & \frac{1}{4} k \ell m (k \ell+1) \\
 \frac{1}{2} k \ell m (k \ell+1) & \frac{1}{4} \ell (k \ell+1) (k m-1) & \frac{1}{4} \ell (k \ell+1) (k m-1) & \frac{1}{4} m \left(k^2 \ell^2-1\right) & \frac{1}{4} k \ell m (k \ell+1) \\
 0 & \frac{1}{2} (k \ell+1)-1 & \frac{1}{2} (k \ell+1) & 0 & 0 \\
 0 & \frac{1}{2} (\ell-1) (k \ell+1) & \frac{1}{2} (\ell-1) (k \ell+1) & 0 & 0 \\
\end{array}
\right),
\end{align*}
and thus the eigenmatrices are determined by \cite[Theorem 4.1 (ii)]{BI} as
\begin{align*}
P&=\left(
\begin{array}{ccccc}
 1 & \frac{k \ell m (k \ell+1)}{2}  & \frac{k \ell m (k \ell+1)}{2}  & k \ell & (\ell-1) (k \ell+1) \\
 1 & 0 & 0 & k \ell & -k \ell-1 \\
 1 & -\frac{\ell (k \ell+1)}{2}  & -\frac{\ell (k \ell+1)}{2}  & k \ell & (\ell-1) (k \ell+1) \\
 1 & -\frac{\sqrt{-m} (k \ell+1)}{2}  & \frac{\sqrt{-m} (k \ell+1)}{2}  & -1 & 0 \\
 1 & \frac{\sqrt{-m} (k \ell+1)}{2}  & -\frac{\sqrt{-m} (k \ell+1)}{2}  & -1 & 0 \\
\end{array}
\right),\\
Q&=\left(
\begin{array}{ccccc}
 1 & (\ell-1) (k m+1) & k m & \frac{k \ell^2 (k m+1)}{2}  & \frac{k \ell^2 (k m+1)}{2}  \\
 1 & 0 & -1 & -\frac{\sqrt{-m}\ell (k m+1)}{2 m} & \frac{\sqrt{-m}\ell (k m+1)}{2 m} \\
 1 & 0 & -1 & \frac{\sqrt{-m}\ell (k m+1)}{2 m} & -\frac{\sqrt{-m}\ell (k m+1)}{2 m} \\
 1 & (\ell-1) (k m+1) & k m & -\frac{\ell (k m+1)}{2}  & -\frac{\ell (k m+1)}{2}  \\
 1 & -k m-1 & k m & 0 & 0 \\
\end{array}
\right).
\end{align*}
(2) The intersection matrix $L_1=(p_{1,j}^k)_{j,k=0}^{3}$ is given as
\begin{align*}
B_1=\left(
\begin{array}{cccc}
 0 & 1 & 0 & 0 \\
 0 & \frac{1}{4} (k+1) (k m-1) & \frac{1}{4} (k+1) (k m-1) & \frac{1}{4} (k+1)^2 m \\
 \frac{1}{2} k (k+1) m & \frac{1}{4} (k+1) (k m-1) & \frac{1}{4} (k+1) (k m-1) & \frac{1}{4} \left(k^2-1\right) m \\
 0 & \frac{k-1}{2} & \frac{k+1}{2} & 0 \\
\end{array}
\right),
\end{align*}
and thus the eigenmatrices are determined by \cite[Theorem 4.1 (ii)]{BI} as
\begin{align*}
P&=\left(
\begin{array}{cccc}
 1 & \frac{1}{2} k (k+1) m & \frac{1}{2} k (k+1) m & k \\
 1 & -\frac{1}{2} (k+1) \sqrt{-m} & \frac{1}{2}  (k+1) \sqrt{-m} & -1 \\
 1 & \frac{1}{2}  (k+1) \sqrt{-m} & -\frac{1}{2}  (k+1) \sqrt{-m} & -1 \\
 1 & \frac{1}{2} (-k-1) & \frac{1}{2} (-k-1) & k \\
\end{array}
\right),\\
Q&=\left(
\begin{array}{cccc}
 1 & \frac{1}{2} k (k m+1) & \frac{1}{2} k (k m+1) & k m \\
 1 & \frac{\sqrt{-1} (k m+1)}{2 \sqrt{m}} & -\frac{\sqrt{-1} (k m+1)}{2 \sqrt{m}} & -1 \\
 1 & -\frac{\sqrt{-1} (k m+1)}{2 \sqrt{m}} & \frac{\sqrt{-1} (k m+1)}{2 \sqrt{m}} & -1 \\
 1 & \frac{1}{2} (-k m-1) & \frac{1}{2} (-k m-1) & k m \\
\end{array}
\right).
\end{align*}
\begin{remark}
\begin{enumerate}
\item The commutative association schemes with $3$ classes in Theorem~\ref{thm:as1} (2) are examples of \cite[Theorem~5.5 (4)]{GC}.
\item The adjacency matrix $A_1$ has the eigenvalues $\frac{k \ell m (k \ell+1)}{2},0,-\frac{\ell (k \ell+1)}{2},\pm \frac{\sqrt{-m} (k \ell+1)}{2}$ and is normal. The upper bound of the size of cocliques in the regular graph whose adjacency matrix $A_1$ was shown in \cite{KS17} to be
$$
\frac{n(-\theta_{\min})}{k_1-\theta_{\min}}=\ell(k\ell+1),
$$
where $n$ is the number of vertices, $k_1$ is the valency, and $\theta_{\min}=\min\{\text{Re}(\theta)\mid \theta \text{ is an eigenvalue of }A_1\}$.
The cocliques described $A_3+A_4$ attain this upper bound.
\end{enumerate}
\end{remark}

\vspace*{0.4cm}

\noindent {\bf Acknowledgements} \\ 
The authors thank the referees for their comments. Hadi Kharaghani acknowledges the support of the Natural Sciences and Engineering Research Council of Canada (NSERC).
Sho Suda is supported by JSPS KAKENHI Grant Number 18K03395.


\begin{thebibliography}{99}
\bibitem{BI}
E. Bannai, T. Ito, Algebraic Combinatorics I: Association Schemes,
{Benjamin/Cummings, Menlo Park, CA,} 1984.


\bibitem{c}
 Robert William Craigen,  Constructions for orthogonal matrices, Thesis (Ph.D.)-University of Waterloo (Canada). 1991. 220 pp.

\bibitem{D}
E. van Dam, 
Three-class association schemes, 
{\sl J. Algebraic Combin.} {\bf 10} (1999), 69--107.

\bibitem{DGS}
P. Delsarte, J. M. Goethals, J. J. Seidel, Spherical codes and designs,
{\sl Geom. Dedicata} {\bf 6} (1977), 363--388.


\bibitem{gs-72}
A. V. Geramita,  J. Seberry,  
Orthogonal designs. Quadratic forms and Hadamard matrices. Lecture Notes in Pure and Applied Mathematics, 45. Marcel Dekker, Inc., New York, 1979.

\bibitem{GC}
R. W. Goldbach, H. L. Claasen, 
The structure of imprimitive non-symmetric $3$-class association schemes,  
{\sl European J.\ Combin.\ } {\bf 17} (1996), 23--37. 

\bibitem{HKT}
W. H. Holzmann, H. Kharaghani, B. Tayfeh-Rezaie, 
All triples for orthogonal designs of order $40$, 
{\sl Discrete Math.} {\bf 308} (2008),  2796--2801.


\bibitem{K}
H. Kharaghani, 
New class of weighing matrices, 
{\sl Ars.\ Combin.} {\bf 19} (1985), 69--72. 

\bibitem{KS17}
H. Kharaghani, S. Suda,
Hoffman's coclique bound for normal regular digraphs, and nonsymmetric association schemes, 
{\sl Mathematics Across Contemporary Sciences}, 137--150, 
Springer Proc. Math. Stat., 190, Springer, Cham, 2017.  

\bibitem{KS2017}
H. Kharaghani and S. Suda, 
Linked systems of symmetric group divisible designs, {\sl J. Algebraic Combin.}  {\bf 47} (2017), no. 2, 319--343.

\bibitem{KS2019}
H. Kharaghani and S. Suda, Linked system of symmetric group divisible designs of type II, {\sl Des.\ Codes  Cryptogr.} {\bf 87} (2019), no. 10, 2341--2360. 


\bibitem{KS19}
H. Kharaghani, S. Suda, 
Linked systems of symmetric group divisible designs of type II,
{\sl Des. Codes Cryptogr.} {\bf 87} (2019), 2341--2360.

\bibitem{M}
R. Mathon, 
The systems of linked $2$-$(16,6,2)$ designs, 
{\sl Ars Combin.} {\bf 11} (1981), 131--148. 


\end{thebibliography}
\end{document}